\documentclass[11pt]{article}

\usepackage{amsfonts,amssymb,amsthm,amsmath}
\usepackage{hyperref}
\usepackage{setspace}
\usepackage{tikz-cd}
\usepackage{fancyhdr}
\newtheorem{definition}{Definition}[section]
\newtheorem{conjecture}[definition]{Conjecture}
\newtheorem{lemma}[definition]{Lemma}
\newtheorem{proposition}[definition]{Proposition}
\newtheorem{theorem}[definition]{Theorem}

\title{On a conjecture of Braverman and Kazhdan}
\author{S. Cheng, B. C. Ng\^o}
\date{}
\begin{document}
\maketitle

\begin{abstract}
In this paper a proof of Conjecture 9.12 of Braverman--Kazhdan in their article \emph{$\gamma$-functions of representations and lifting} on the acyclicity of their $\ell$-adic $\gamma$-sheaves over certain affine spaces is given for $\mathrm{GL}(n)$.
\end{abstract}

\section*{Introduction}
\addcontentsline{toc}{section}{Introduction}

Classical Fourier transforms on vector spaces over local fields and adelic rings have found remarkable connections with the standard $L$-functions $L(s,\pi,\mathrm{std})$ of $\mathrm{GL}(n)$ since Tate \cite{T50} for $n=1$ and Godement--Jacquet \cite{GJ72} for general $n$.

More generally, for each reductive group $G$ which is quasi-split over a nonarchimedean local field $F_v$ and each representation $\rho$ of its dual group $^LG$ satisfying some mild technical conditions, there exists a $\rho$-analogue of the Fourier transform which is essentially the operator of convolution by an invariant distribution $\varPhi_{\psi,\rho}$ on $G(F_v)$ where $\psi$ is a fixed additive character of $F_v$, whose operator-valued Mellin transform $\mathcal{M}(\varPhi_{\psi,\rho}*(\bullet))$ is the scalar operator of multiplication by the $\gamma$-function $\gamma_{\psi,\rho}({\pi_v})$ investigated by Braverman--Kazhdan in \cite{BK00}. Similar ideas have been developed further by L. Lafforgue, see for example \cite{L13}.

Incarnations $\Phi_{\psi,\rho}$ of $\varPhi_{\psi,\rho}$ as $\ell$-adic perverse sheaves over finite fields have been constructed and studied by Braverman--Kazhdan in the last section of \cite{BK00} and subsequently in \cite{BK02}. The purpose of this paper is to establish Conjecture 9.12 in \cite{BK00} for $\mathrm{GL}(n)$. The argument generalizes that of Braverman--Kazhdan in \cite{BK02} for $\mathrm{GL}(2)$. Following the classical paradigm the generalization from $\mathrm{GL}(2)$ to $\mathrm{GL}(n)$ involves mirabolic groups as an essential ingredient.

\paragraph*{Conventions}
In this paper $k$ is an algebraic closure of a finite field $k_0$ with $q$ elements of characteristic $p$. Let $\ell$ be a prime number which is distinct from $p$, let $\overline{\mathbb{Q}}_\ell$ be an algebraic closure of the field of $\ell$-adic numbers.

If $X$ is a $k$-scheme, let $D_c^b(X)$ denote the derived category of complexes of $\overline{\mathbb{Q}}_\ell$-\'{e}tale sheaves on $X$ with bounded constructible cohomology, let $[d]$ denote the $d$th translation functor on $D_c^b(X)$. If $f$ is a $k$-linear morphism of $k$-schemes, the six functors $f^*$, $f_*$, $f_!$, $f^!$, $\otimes_X$ and $\mathcal{H}om_X$ are understood in the derived sense. If $j$ is the morphism of inclusion of an open $k$-subscheme, let $j_{!*}$ denote the intermediate extension functor of Goresky--MacPherson for $\overline{\mathbb{Q}}_\ell$-perverse sheaves (see \cite{BBD82}).

We will denote by $\mathbb{G}_a$ the additive group defined over $k$. It has the Artin--Schreier covering, which is a torsor under the finite group $k_0$, given by the Lang isogeny $L_{\mathbb{G}_a}:\mathbb{G}_a\to \mathbb{G}_a$ with $L_{\mathbb{G}_a}(t)= t^q-t$. We fix a nontrivial character $\psi:k_0\to \overline{\mathbb{Q}}_\ell^\times$ and denote by $\mathcal{L}_\psi$ the rank one $\overline{\mathbb{Q}}_\ell$-local system attached to $\psi$ obtained by pushing out the Artin--Schreier covering.

Similarly, we denote by $\mathcal{L}_c$ the the rank one $\overline{\mathbb{Q}}_\ell$-local system on the multiplicative group $\mathbb{G}_m$ attached to a multiplicative character $c:k_0^\times\to\overline{\mathbb{Q}}_\ell^\times$ obtained by pushing out the Kummer covering $L_{\mathbb{G}_m}:\mathbb{G}_m\to \mathbb{G}_m$ with $L_{\mathbb{G}_m}(t)=t^{q-1}$.

\tableofcontents

\section{Katz's hypergeometric sheaves}\label{section hypergeometric sheaves}

Let $T$ be a torus defined over $k$ and $\lambda:\mathbb{G}_m \to T$ a nontrivial cocharacter. We note that $\lambda$ is then necessarily a finite morphism so that we have $$
	\Psi(\lambda)=\lambda_! ( j^*\mathcal{L}_\psi[1])=\lambda_* (j^*\mathcal{L}_\psi[1]).
$$ where $j:\mathbb{G}_m\to \mathbb{G}_a$ is the inclusion morphism from the multiplicative group into the additive group. We will call $\Psi(\lambda)$ the hypergeometric sheaf attached to $\lambda:\mathbb{G}_m\to T$. They are perverse sheaves on $T$.

Following Katz \cite[Chapter 8]{K90}, we construct general hypergeometric sheaves on $T$ by convolving the $\Psi(\lambda)$. We recall that convolution products on $T$ are constructed as direct images with respect to the multiplication morphism $\mu:T\times T \to T$ with $\mu(t_1,t_2)=t_1 t_2$. There are two convolution products attached to the direct image functors with or without the compact support condition: for every $\mathcal{F},\mathcal{G}\in D_c^b(T)$ we define
\begin{eqnarray*}
	\mathcal{F}\star\mathcal{G} &=& \mu_! (\mathcal{F}\boxtimes \mathcal{G})\\
	\mathcal{F}*\mathcal{G} &=& \mu_*(\mathcal{F}\boxtimes \mathcal{G})
\end{eqnarray*}
related by the morphism of functors $$
	\mathcal{F}\star\mathcal{G} \to \mathcal{F}*\mathcal{G}
$$ that consists in forgetting the compact support condition. For every collection of possibly repeated nontrivial cocharacters $\underline\lambda=(\lambda_1,\ldots,\lambda_r)$, we consider the convolution products 
\begin{eqnarray*}
	\Psi_{\underline\lambda}&=&\Psi(\lambda_1) \star \dots \star \Psi(\lambda_r) \\
	\Psi^*_{\underline\lambda}&=&\Psi(\lambda_1) * \dots * \Psi(\lambda_r)
\end{eqnarray*}
and the forget support morphism
\begin{equation} \label{forget-support}
    \Psi_{\underline\lambda} \to \Psi^*_{\underline\lambda}.
\end{equation}

If we denote by ${\mathrm p}_{\underline\lambda}:\mathbb{G}_m^r \to T$ the homomorphism given by $$
	{\mathrm p}_{\underline\lambda}(t_1,\ldots,t_r)=\prod_{i=1}^r \lambda_i(t_i)
$$ and by $\mathrm{tr}:\mathbb{G}_m^r \to \mathbb{G}_a$ the addition morphism $\mathrm{tr}(t_1,\ldots,t_r)=\sum_{i=1}^r t_i$, then
\begin{eqnarray*}
	\Psi_{\underline\lambda} & = & {\mathrm p}_{\underline\lambda,!} \mathrm{tr}^* \mathcal{L}_\psi [r]\\
	\Psi^*_{\underline\lambda} & = & {\mathrm p}_{\underline\lambda,*} \mathrm{tr}^* \mathcal{L}_\psi [r].
\end{eqnarray*}
For $T=\mathbb{G}_m$ and $\lambda_i:\mathbb{G}_m\to \mathbb{G}_m$ being the identity for all $i$, $\Psi_{\underline\lambda}$ is the $r$-fold Kloosterman sheaf considered by Deligne in \cite{D77s}. In general, this is what Braverman and Kazhdan have called $\gamma$-sheaves on tori in \cite{BK02}.

We will restrict ourselves in a setting where the morphism \eqref{forget-support} is an isomorphism. Let $\sigma:T\to \mathbb{G}_m$ be a character. A cocharacter $\lambda:\mathbb{G}_m\to T$ is said to be $\sigma$-positive if the composition $\sigma\circ\lambda:\mathbb{G}_m\to\mathbb{G}_m$ is of the form $t\mapsto t^n$ where $n$ is a positive integer. 

\begin{proposition}\label{positive-hypergeometric-sheaves}
	Assume that $\lambda_1,\ldots,\lambda_n$ are $\sigma$-positive. Then the forget-support morphism $\Psi_{\underline\lambda} \to \Psi^*_{\underline\lambda}$ is an isomorphism. Moreover, $\Psi_{\underline\lambda}$ is a perverse local system over the image of ${\mathrm p}_{\underline\lambda}$, which is a subtorus of $T$.
\end{proposition}

\begin{proof}
See Appendix \ref{appendix-positive-hypergeometric-sheaves}.
\end{proof}

Let $\Sigma_{\underline\lambda}$ denote the subgroup of the symmetric group $\mathfrak{S}_r$ consisting of permutations $\tau\in \mathfrak{S}_r$ such that for all $i\in\{1,\ldots,r\}$, we have $\lambda_{\tau(i)}=\lambda_i$. This subgroup is of the form $\Sigma_{\underline\lambda}=\mathfrak{S}_{r_1}\times\cdots\times\mathfrak{S}_{r_m}$ where $(r_1,\ldots,r_m)$ is the partition of $r$ corresponding to positive number of occurrences in $\{\lambda_1,\ldots,\lambda_r\}$.  

\begin{proposition}
	The group $\Sigma_{\underline\lambda}$ acts on $\Psi_{\underline\lambda}$ via the sign character.
\end{proposition}

\begin{proof}
This is \cite[Proposition 7.20]{D77s}.
\end{proof}

\begin{proposition}
Let $\mathcal{L}$ be a Kummer local system on $T$. Then if $\lambda_1,\ldots,\lambda_n$ are $\sigma$-positive, we have $$
	\mathrm{H}^i_c(\Psi_{\underline\lambda} \star \mathcal{L})=0
$$ for $i\neq 0$ and $\dim \mathrm{H}^0_c(\Psi_{\underline\lambda} \star \mathcal{L})=1$. Moreover, there is a canonical isomorphism $$
	\Psi_{\underline\lambda} \star \mathcal{L} = \mathrm{H}^0_c(\Psi_{\underline\lambda}\star \mathcal{L}) \otimes\mathcal{L}.$$ 
 
\end{proposition}

\begin{proof}
This is \cite[Theorem 4.8]{BK02}.
\end{proof}

\section{Braverman--Kazhdan's $\gamma$-sheaves}

Let $G$ be a reductive group over $k$.
Let $T$ be a maximal torus of $G$, $B$ a Borel subgroup containing $T$ and $U$ the unipotent radical of $B$. Let $W=\mathrm{Nor}_G(T)/T$ denote the Weyl group of $G$, $\mathrm{Nor}_G(T)$ being the normalizer of $T$ in $G$. The group of cocharacters $\Lambda=\mathrm{Hom}(\mathbb{G}_m,T)$ is a free abelian group of finite type equipped with an action of $W$. 
The complex dual group $\check G$ is equipped with a maximal torus $\check T$ and a Borel subgroup $\check B$ containing $\check T$. We have $\Lambda=\mathrm{Hom}(\check T,\mathbb{C}^\times)$. 

We will recall the construction, due to Braverman and Kazhdan, of the $\gamma$-sheaf attached to a representation of the dual group $\check G$.
Let $\rho:\check G\to \mathrm{GL}(V_\rho)$ be an $r$-dimensional representation of $\check G$. The restriction of $\rho$ to $\check T$ is diagonalizable i.e. there exists a finite set of weights $$
	\{\lambda_1,\ldots,\lambda_m\} \subset \Lambda=\mathrm{Hom}(\check T,\mathbb{C}^\times)
$$ such that there is a decomposition into direct sum of eigenspaces $$
	V_\rho=\bigoplus_{i=1}^m V_{\lambda_i},
$$ with $\check T$ acting on $V_{\lambda_i}$ by the character $\lambda_i$. The integers $r_i=\dim(V_{\lambda_i})$ define a partition $r=r_1+\cdots+r_m$.
We will denote $$
	\underline \lambda=(\underbrace{\lambda_1,\ldots,\lambda_1}_{r_1},\ldots,\underbrace{\lambda_m,\ldots,\lambda_m}_{r_m}) \in \Lambda^r
$$ where $\lambda_1,\ldots,\lambda_m$ appear in $\underline\lambda$ with multiplicity $r_1,\ldots,r_m$ respectively.

By choosing a basis $A_i=\{v_{i,j},1\leq j\leq r_i\}$ of each $V_{\lambda_i}$, we obtain a basis $$
	A=A_1 \sqcup \ldots \sqcup A_m
$$ of $V_\rho$. The Weyl group of $\mathrm{GL}(V_\rho)$ can be identified with the symmetric group ${\rm Perm}(A)=\mathfrak{S}_r$ of permutations of the finite set $A$.  
Let 
$$\Sigma_{\underline\lambda}=\mathfrak{S}_{r_1}\times \cdots \times \mathfrak{S}_{r_m}\subset \mathfrak{S}_r$$
denote the subgroup consisting of $\tau\in{\rm Perm}(A)$ such that $\tau(A_i)=A_i$. 

Let $\Sigma'_{\underline\lambda}$ denote the subgroup of ${\rm Perm}(A)$ consisting of permutations $\tau$ such that there exists a permutation $\xi\in\mathfrak{S}_r$ such that 
$\tau(A_i)=A_{\xi(i)}$ for all $i\in \{1,\ldots,m\}$. The application $\tau\mapsto \xi$ defines a homomorphism  $\Sigma'_{\underline\lambda}\to \mathfrak{S}_m$ whose kernel is $\Sigma_{\underline\lambda}$. Its image consists of permutations $\xi\in \mathfrak{S}_m$ preserving the function $i\mapsto r_i$. 

The Weyl group $W$ operates on $\Lambda$ and its action preserves the subset $\{\lambda_1,\ldots,\lambda_m\}$ of $\Lambda$. It induces a homomorphism $W\to \mathfrak{S}_m$. Its image is contained in the subgroup of $\mathfrak{S}_m$ of permutations preserving the function $i\mapsto r_i$ so that there is a canonical homomorphism $$
	\rho_W:W\to \Sigma'_{\underline\lambda}/\Sigma_{\underline\lambda}.
$$ We derive an extension $W'$ of $W$ by $\Sigma_{\underline\lambda}$ fitting into the diagram
$$\begin{tikzcd}
1 \arrow{r} & \Sigma_{\underline\lambda} \arrow{r}\arrow{d} & W' \arrow{d}\arrow{r} & W \arrow{r}\arrow{d} & 1 \\
1 \arrow{r} & \Sigma_{\underline\lambda} \arrow{r} & \Sigma'_{\underline\lambda} \arrow{r} & \Sigma'_{\underline\lambda}/\Sigma_{\underline\lambda} \arrow{r} & 1
\end{tikzcd}$$
where an element $w'\in W'$ consists of a pair $(w,\xi)$ with $w\in W$ and $\xi\in \Sigma'_{\underline\lambda}$ such that $\rho_W(w)= \xi \mod \Sigma_{\underline\lambda}$. One can check that the homomorphism $\mathrm{p}_{\underline\lambda}: \mathbb{G}_m^r \to T$, and its dual $\rho|_{T'}: \check T \to (\mathbb{C}^\times)^r$, are $W'$-equivariant.

As in Section \ref{section hypergeometric sheaves}, the finite sequence of $\sigma$-positive weights $\underline \lambda\in \Lambda^r$ gives rise to a hypergeometric sheaf $\Psi_{\underline\lambda}$ on $T$, equipped with an action of $\Sigma_{\underline\lambda}$. This hypergeometric sheaf is well behaved under certain positivity condition that can be phrased in the present circumstance as follows.  
Let $\sigma:G\to \mathbb{G}_m$ be a character of $G$, we also denote $\sigma:\mathbb{C}^\times\to \check G$ the dual cocharacter. A representation $\rho:\check G\to \mathrm{GL}(V_\rho)$ is said to be $\sigma$-positive if for every weight $\lambda_i$ occurring $V_\rho$, $\lambda_i\circ \sigma:\mathbb{C}^\times\to \mathbb{C}^\times$ is of the form $t\mapsto t^n$ where $n$ is a positive integer.  We will assume that $\rho$ is $\sigma$-positive.  We will also assume that the homomorphism $\mathrm{p}_{\underline \lambda}:\mathbb{G}_m^r\to T$ is surjective. Under these assumptions, we know that $\Psi_{\underline\lambda}$ is a local system on $T$ with the degree shift $[\dim(T)]$. 

As the homomorphism $\mathrm{p}_{\underline\lambda}: \mathbb{G}_m^r \to T$ is $W'$-equivariant and the morphism $\mathrm{tr}:\mathbb{G}_m^r\to \mathbb{G}_a$ is invariant under the action of $W'$, we have an action of $W'$ on the hypergeometric sheaf $\Psi_{\underline\lambda}=\mathrm{p}_{\underline\lambda,!} \mathrm{tr}^* \mathcal{L}_\psi [r]$ compatible with the action of $W'$ on $T$ via $W'\to W$: for every $w'=(w,\xi)$ with $w\in W$ and $\xi\in \Sigma'_{\underline\lambda}$ having the same image in $\Sigma'_{\underline\lambda}/\Sigma_{\underline\lambda}$, by \cite[Proposition 6.2]{BK02} we have an isomorphism, $$
	\iota'_{w'}: w^* \Psi_{\underline\lambda} \to \Psi_{\underline\lambda}.
$$ We also know that the restriction of this action to $\Sigma_{\underline\lambda}$ is the sign character i.e. for $w'=(1,\xi)$ with $\xi\in \Sigma_{\underline\lambda}$, we have $\iota'_{w'}=\mathrm{sign}(\xi)$, $\xi$ being considered as a permutation of the finite set $A$. For every $w'=(w,\xi)$ we set  
\begin{equation}\label{action-of-W}
	\iota_{w'}=\mathrm{sign}_r(\xi)\mathrm{sign}_W(w) \iota'_{w'}: w^* \Psi_{\underline\lambda} \to \Psi_{\underline\lambda}
\end{equation}
where $\mathrm{sign}_r:\mathfrak{S}_r\to \{\pm 1\}$ and $\mathrm{sign}_W:W\to \{\pm 1\}$ are the sign characters of $\mathfrak{S}_r$ and $W$ respectively. We can then check that $\iota_{w'}$ depends only on $w$ so that we get an action of $W$ on $\Psi_{\underline\lambda}$.

Now we recall the Grothendieck--Springer simultaneous resolution of the fibers of the Steinberg morphism (see \cite[Section 6]{S65}) $c:G\to S=T/W$: $$
\begin{tikzcd}
\tilde G \arrow{r}{\tilde c} \arrow{d}[swap]{\tilde q}
& T \arrow{d}{q} \\
G \arrow{r}[swap]{c}
& S
\end{tikzcd}
$$ where $\tilde G$ is the variety of pairs $(g,h)\in G\times G/B$ such that $h^{-1}gh\in B$, the morphism $\tilde c$ given by $(g,h) \mapsto h^{-1}gh \mod U$ is smooth, and the morphism $\tilde q$ given by $(g,h)\mapsto g$ is proper and small in the sense of Goresky--MacPherson.  
If $T^\mathrm{rss}$ is the largest open subset of $T$ where $W$ acts freely and $S^\mathrm{rss}=T^\mathrm{rss}/W$, then the diagram is Cartesian over $S^\mathrm{rss}$. In particular $\tilde G^\mathrm{rss} \to G^\mathrm{rss}$ is a $W$-torsor, where $j^\mathrm{rss}:G^\mathrm{rss}\to G$ denotes the base change of the inclusion morphism $S^\mathrm{rss}\subset S$ to $G$.

Recall that the induction functor $\mathrm{Ind}_T^G: D^b_c(T) \to D^b_c(G)$ is defined by $$
	\mathrm{Ind}_T^G(\mathcal{F})=\tilde q_! \tilde c^* \mathcal{F}[d]
$$ where $d=\dim(G)-\dim(T)$. Because $q$ is a small map and $\Psi_{\underline\lambda}$ is a perverse local system on $T$,
$$	\mathrm{Ind}_T^G(\Psi_{\underline\lambda})=\tilde q_! \tilde c^* \Psi_{\underline\lambda} [d]$$
is a perverse sheaf isomorphic to the intermediate extension of its restriction to $G^\mathrm{rss}$:
$$	\mathrm{Ind}_T^G(\Psi_{\underline\lambda})=j^\mathrm{rss}_{!*} j^{\mathrm{rss},*}
	\mathrm{Ind}_T^G(\Psi_{\underline\lambda}).$$
For $\tilde G^\mathrm{rss} \to G^\mathrm{rss}$ is a $W$-torsor and $\Psi_{\underline\lambda}$ is $W$-equivariant, $W$ operates on $j^{\mathrm{rss},*}\mathrm{Ind}_T^G(\Psi_{\underline\lambda})$. By functoriality of the intermediate extension functor, this action of $W$ can be extended to the perverse sheaf $\mathrm{Ind}_T^G(\Psi_{\underline\lambda})$. 

\begin{definition}
	The $\gamma$-sheaf attached to $\rho$ is the $W$-invariant direct factor of the perverse sheaf $\mathrm{Ind}_T^G(\Psi_{\underline\lambda})$
$$		\Phi_\rho=\mathrm{Ind}_T^G(\Psi_{\underline\lambda})^W.$$
\end{definition}

As a direct factor of $\mathrm{Ind}_T^G(\Psi_{\underline\lambda})$, $\Phi_\rho$ is also isomorphic to the intermediate extension of its restriction to $G^\mathrm{rss}$. There is thus a slightly different way to construct it: we start by descending the restriction of  $\Psi_{\underline\lambda}$ to $T^\mathrm{rss}$, which is a $W$-equivariant perverse local system on $T^\mathrm{rss}$, to a perverse local system $\Phi_{\underline\lambda,S^\mathrm{rss}}$ on $S^\mathrm{rss}$. Then we have
$$	\Phi_\rho = j^\mathrm{rss}_{!*} c^{\mathrm{rss},*} \Phi_{\underline\lambda,S^\mathrm{rss}}[d].$$

The induction functor admits a left adjoint $\mathrm{Res}_T^G:D_c^b(G)\to D_c^b(T)$, the restriction functor, defined by
\begin{equation}\label{restriction definition}
\mathrm{Res}_T^G(\mathcal{F})=\pi_!i^*(\mathcal{F})
\end{equation}
where $\pi:B\to B/U=T$ and $i:B\to G$ denote the quotient and inclusion morphisms. More generally $\mathrm{Res}_M^G:D_c^b(G)\to D_c^b(M)$ could be defined if we replace $B$ by a standard parabolic subgroup $P$ of $G$ and $T$ by the Levi component $M$ of $P$ in \eqref{restriction definition}. The adjunction between restriction and induction and Frobenius reciprocity imply the following

\begin{proposition}\label{restriction lemma}
Let $\Phi_{\rho,M}$ denote the perverse sheaf $\Phi_{\rho'}$ on $M$ where $\rho'$ denotes the restriction of $\rho$ from $\check{G}$ to $\check{M}$, then
$$\Phi_{\rho,M}\simeq\mathrm{Res}_M^G(\Phi_\rho).$$
\end{proposition}

\begin{proof}
This is the first statement of \cite[Theorem 6.6]{BK02}.
\end{proof}

In \cite[Conjecture 9.2]{BK00} Braverman and Kazhdan have conjectured the following vanishing property of $\Phi_\rho$.

\begin{conjecture}
	Let $\pi_U:G\to G/U$ denote the quotient map. For every $\sigma$-positive representation $\rho$ of $\check G$, $\pi_{U,!} \Phi_\rho$ is supported on the closed subset $T=B/U$ of $G/U$.
\end{conjecture}

The conjecture can be reformulated as follows. For every geometric point $g\in G-B$, we conjecture that 
$$\mathrm{H}_{c}^*(gU,i^*\Phi_\rho)=0$$
where $i:gU\to G$ denotes the inclusion map.

Braverman and Kazhdan have verified their conjecture for groups of semisimple rank one, and for $G=\mathrm{GL}(n)$, $\sigma=\det$ and $\rho$ the standard representation of $\mathrm{GL}_n(\mathbb{C})$.

\begin{theorem} \label{BKGL}
	The above conjecture holds for $G=\mathrm{GL}(n)$, $\sigma=\det$ and an arbitrary  $\sigma$-positive  representation of $\rho$ of $\mathrm{GL}_n(\mathbb{C})$.
\end{theorem}

Our argument applies in the case when $\mathrm{p}_{\underline{\lambda}}:\mathbb{G}_m^r\to T$ is surjective, however for $\sigma$-positive $\rho$ the only other possibility is when $\rho$ factors through $\mathrm{det}:\mathrm{GL}_n(\mathbb{C})\to\mathbb{C}^\times$, which implies that $\Phi_\rho$ is supported on the center of $\mathrm{GL}(n)$, so the theorem holds trivially in this case as well.

\section{Conjugation action of the mirabolic group}
\label{mirabolic}

Our proof of the Braverman--Kazhdan conjecture for $G=\mathrm{GL}(n)$ is based on the geometry of the conjugation action of the mirabolic. This geometry has been described by Bernstein in \cite[4.1-4.2]{B84}. What we aim for here is to put Bernstein's description into a form suitable for our purpose. For consistency of notations with \cite{B84} we will follow Bernstein and consider left group actions and left quotients for the rest of this paper.  In particular our version of Theorem \ref{BKGL} applies to $U\backslash G$ instead of $G/U$.

Let $V$ be the standard $n$-dimension $k$-vector space with the standard basis $e_1,\ldots,e_n$. We consider the filtration $0\subset F_1 \subset F_2 \subset \cdots \subset F_n=V$ where $F_i$ is the subvector space generated by $e_1,\ldots,e_i$. We will denote $E_i=V/F_i$. 

We denote by $Q$ the mirabolic subgroup of $G=\mathrm{GL}(V)$ consisting of elements $g\in G$ fixing the line generated by $e_1$, and $Q_1$ the subgroup consisting of elements $g\in G$ fixing the vector $e_1$.
We consider the $Q$-equivariant stratification of $G$
\begin{equation} \label{stratification}
	G=\bigsqcup_{m=1}^n X_m
\end{equation} 
where $X_m$ is the locally closed subset of $V$ consisting of $g\in G$ such that the subspace $F_x$ of $V$ generated by the vectors $v,xv,x^2v,\ldots$ is of dimension $m$. We will prove that $[Q_1\backslash X_m]$ ``looks like'' a similar quotient $[Q_1\backslash G]$ in lower rank. This statement can be made precise as follows.   

\begin{theorem} \label{Bernstein-stratification}
	There exists a smooth surjective morphism 
$$\phi_m:[\mathrm{GL}(n-m)\backslash(\mathbb{A}^m \times \mathrm{GL}(n-m) \times \mathbb{A}^{n-m})] \to [Q_1\backslash X_m]$$
where $g\in \mathrm{GL}_{n-m}$ acts on $(a,x',v)\in \mathbb{A}^m \times \mathrm{GL}(n-m) \times \mathbb{A}^{n-m}$ by the formula 
$$g(a,x',v) =(a,gx'g^{-1},vg^{-1}).$$ 
Moreover, if $\phi_m(a,x',v)=x$ then 
$$c(x)=a_tc(x')$$
where $c(x)$ and $c(x')$ are the characteristic polynomials of $x$ and $x'$ respectively, written with the formal variable $t$, and where $a_t$ is the polynomial $$a_t=t^m+a_1 t^{m-1} +\cdots + a_m$$ for every $a=(a_1,\ldots,a_m) \in \mathbb{A}^m$.\end{theorem}

Theorem \ref{Bernstein-stratification} implies, through an induction on the rank $n$, that the mirabolic group acts on the space of matrices $x\in G$ of a given characteristic polynomial with finitely many orbits. 

\begin{proof}
For each $m$, we have a fibration
$$\pi_m: X_m \to \mathrm{Gr}_m$$
where $\mathrm{Gr}_m$ is the Grassmannian of $m$-dimensional subspaces $F_x$ of $V$ containing $v$. It maps $x\in X_m$ to the subspace $F_x$ of $V$ generated by the vectors $v,xv,x^2v,\ldots$ which is of dimension $m$. 

For each $m$ we also have the subspace $F_m$ and the quotient space $E_m$ of $V$.
When there is no possible confusion on $m$, we will drop the index $m$ and write simply $F$ for $F_m$, $E$ for $E_m$. 
We consider the parabolic subgroup $P$ of $G$ consisting of elements $g\in G$ such that $gF_m=F_m$. An element $g\in P$ can be described as a block matrix
$$g=\left[	\begin{array}{cc}
		g_F & v \\
		0 & g_E
	\end{array} \right]$$
where $g_F\in \mathrm{GL}(F)$ and $g_E\in \mathrm{GL}(E)$. The group $Q_1$ acts transitively on $\mathrm{Gr}_m$ and its stabilizer at a point $F\in \mathrm{Gr}_m$ is $P\cap Q_1$. The intersection $P\cap Q_1$ can be described by the condition $g_F\in Q_{F,1}$ where $Q_{F,1}$ is the subgroup of $\mathrm{GL}(F)$ defined by $g_F e_1=e_1$. 

The fiber $\pi_m^{-1}(F)$ is the open subset of $P$ consisting of block matrices 
$$x=\left[	\begin{array}{cc}
		x_F & y \\
		0 & x_E
	\end{array} \right]$$
such that $v$ is a cyclic vector of $F$ with respect to the action of $x_F$.
The group $P\cap Q_1$ acts on $\pi_m^{-1}(F)$ by conjugation and we have
$$X_m= \pi_m^{-1}(F) \times^{P \cap Q_1} Q_1.$$

For $x_F$ as above, $v,x_Fv,\ldots,x^{m-1}_Fv$ form a basis of $F$. It follows that  there exists a unique $g_F\in Q_{F,1}$ such that $g_F^{-1}x_F g_F$ has the form of a companion matrix
\begin{equation} \label{companion}
	\left[
	\begin{array}{ccc|c}
		0 & \cdots & 0 & -a_m \\ \hline
		  &  &         & -a_{m-1} \\
		  & \mathrm{I}_{m-1} & & \vdots \\
		  & & & -a_1
	\end{array}
	\right]
\end{equation}
where $\mathrm{I}_{m-1}$ is the the identity matrix of size $m-1$, and $a_1,\ldots,a_m$ are the coefficients of the characteristic polynomial of $x_F$. 
It follows that
$$X_m=Y_m \times^{H_m} Q_1$$
where $Y_m$ is the space of matrices of the form
\begin{equation} \label{x}
x=\left[	\begin{array}{cc}
		x_F & y \\
		0 & x_E
	\end{array} \right]
\end{equation}
where $x_F$ is a companion matrix as in \eqref{companion}, $H_m$ is the subgroup of $P$ of matrices of the form
$$h=\left[	\begin{array}{cc}
		\mathrm{I}_F & v \\
		0 & g_E
	\end{array} \right]$$
acting on $Y_m$ by conjugation. The group $H_m$ has the structure of a semidirect product
$$H_m=U_P \rtimes \mathrm{GL}(E)$$
where $U_P$, the unipotent radical of $P$, consists of matrices of the form
$$u=\left[	\begin{array}{cc}
		\mathrm{I}_F & v \\
		0 & \mathrm{I}_E
	\end{array} \right].$$
It will be convenient to regard $u-\mathrm{I}_V$ as a linear application $v\in \mathrm{Hom}(E,F)$.

The action of $U_P$ on $X_F$ can be written down as follows
$$\left[	\begin{array}{cc}
		\mathrm{I}_F & -v \\
		0 & \mathrm{I}_E
	\end{array} \right]
	\left[	\begin{array}{cc}
		x_F & y \\
		0 & x_E
	\end{array} \right]
	\left[	\begin{array}{cc}
		\mathrm{I}_F & v \\
		0 & \mathrm{I}_E
	\end{array} \right]=
	\left[	\begin{array}{cc}
		x_F & y +x_Fv -vx_E  \\
		0 & x_E
	\end{array} \right].$$
In other words, the action of $v\in \mathrm{Hom}(E,F)$ on the variable $y$ consists in a translation by $x_Fv -vx_E$. 

Now, the rank of the linear transformation of $\mathrm{Hom}(E,F)$ given by 
$$v\mapsto x_Fv -vx_E$$ 
depends on the number of common eigenvalues of $x_F$ and $x_E$, in particular it is an isomorphism if $x_F$ and $x_E$ have no common eigenvalues. Thus if $x_F$ and $x_E$ have no common eigenvalues, $U_P$ acts simply transitively on the fiber of $Y_m$ over $(x_F,x_E)$ by conjugation. However, for our purpose, we will need a statement uniform with respect to $(x_F,x_E)$, no matter whether they have common eigenvalues or not.

\begin{lemma} \label{unipotent-conjugation}
Let $0=F_0\subset F_1\subset F_2 \subset \cdots \subset F_{m-1} \subset F$ be a filtration of $F$ with $\dim(F_j)=j$. Let $x_F\in\mathrm{GL}(F)$ be a linear transformation such that $x_F(F_j) \subset F_{j+1}$ and the induced application $F_j / F_{j-1} \to F_{j+1}/F_j$ is an isomorphism for every $j$ in the range $1\leq j \leq m-1$. Let $x_E\in GL(E)$ be an arbitrary linear transformation.

	We consider two subgroups of $U_P=\mathrm{Hom}(E,F)$:
\begin{itemize}
	\item $U_1=\mathrm{Hom}(E,F_1)$. 
	\item $U_{m-1}=\mathrm{Hom}(E,F_{m-1})$.
\end{itemize}
Then the action of $U_1\times U_{m-1}$ on the space of matrices of the form 
$$x=\left[	\begin{array}{cc}
		x_F & y \\
		0 & x_E
	\end{array} \right]$$
given by
$$(u_1,u_{m-1}) x=u_1 u_{m-1} x u_{m-1}^{-1}$$
is simply transitive. 
\end{lemma}

\begin{proof}
	We write 
	\begin{equation*}
		u_1=\left[	\begin{array}{cc}
		\mathrm{I}_F & v_1 \\
		0 & \mathrm{I}_E
	\end{array} \right] \mbox{ and }
	u_{m-1}=\left[	\begin{array}{cc}
		\mathrm{I}_F & v_{m-1} \\
		0 & \mathrm{I}_E
	\end{array} \right]
	\end{equation*}
with $u_1\in \mathrm{Hom}(E,F_1)$ and $u_{m-1}\in \mathrm{Hom}(E,F_{m-1})$. Then we have
\begin{equation*}
	u_1 u_{m-1} x u_{m-1}^{-1} = 
	\left[ \begin{array}{cc}
		x_F & y + v_1 + v_{m-1}x_E - x_F v_{m-1}\\
		0 & x_E
	\end{array} \right]
\end{equation*}
The lemma is now equivalent to saying that the linear application
$$\mathrm{Hom}(E,F_1) \times \mathrm{Hom}(E,F_{m-1}) \to \mathrm{Hom}(E,F)$$
given by $(v_1,v_{m-1}) \mapsto v_1 + v_{m-1}x_E - x_F v_{m-1}$
is an isomorphism. This is equivalent to proving that the map
$$A=\mathrm{Hom}(E,F_{m-1}) \to \mathrm{Hom}(E,F/E_1)=B$$
given by $$\phi(v_{m-1})=v_{m-1}x_E - x_F v_{m-1} \mod \mathrm{Hom}(E,F_1)$$ is an isomorphism. 

We consider the filtration $0\subset A_1 \subset \cdots \subset A_{m-1}=A$ with $A_j=\mathrm{Hom}(E,F_j)$ and the filtration $0\subset B_1 \subset \cdots \subset B_{m-1}=B$ with $B_j=\mathrm{Hom}(E,F_{j+1}/F_1)$. 
We observe that $\phi(A_j)\subset B_j$ for all $j$ and the induced map on the associated graded $A_{j}/A_{j-1} \to B_j/B_{j-1}$ is an isomorphism. 
Indeed the linear application $\phi_E:v_{m-1}\mapsto v_{m-1}x_E$ satisfies $\phi_E(A_j)\subset B_{j-1}$ and hence induces the zero map on the associated graded $A_{j}/A_{j-1} \to B_j/B_{j-1}$. On the other hand, $\phi_F:v_{m-1}\mapsto x_Fv_{m-1}$ satisfies $\phi_F(A_j)\subset B_j$ and induces an isomorphism on the associated graded $A_{j}/A_{j-1} \to B_j/B_{j-1}$ by assumption on $x_F$. It follows that $\phi$ is an isomorphism.
\end{proof}

We infer from the lemma the existence of a canonical isomorphism 
	$$Y_m=\mathbb{A}^m \times \mathrm{GL}(E) \times \mathrm{Hom}(E,F_1) \times \mathrm{Hom}(E,F_{m-1})$$
mapping $x \mapsto (x_F,x_E,v_1,v_{m-1})$ with $x_F$ a companion matrix as in \eqref{companion}, $x_E\in \mathrm{GL}(E)$, $v_1\in \mathrm{Hom}(E,F_1)$ and $v_{m-1} \in \mathrm{Hom}(E,F_{m-1})$ such that 
$$x=	
	\left[	\begin{array}{cc}
		\mathrm{I}_F & v_1 \\
		0 & \mathrm{I}_E
	\end{array} \right] 
	\left[	\begin{array}{cc}
		\mathrm{I}_F & v_{m-1} \\
		0 & \mathrm{I}_E
	\end{array} \right] 
	\left[	\begin{array}{cc}
		x_F & 0 \\
		0 & x_E
	\end{array} \right] \left[	\begin{array}{cc}
		\mathrm{I}_F & -v_{m-1} \\
		0 & \mathrm{I}_E
	\end{array} \right].$$
In these new coordinates, the action of the subgroup $U_{m-1}\rtimes \mathrm{GL}(E)$ of 
$H_m=U_P \rtimes \mathrm{GL}(E)$ can be described as follows: the action $(v'_{m-1},g_E) \in \mathrm{Hom}(E,F_{m-1})\rtimes \mathrm{GL}(E)$ on $x=(x_F,x_E,v_1,v_{m-1})$ is given by:
$$(v'_{m-1},g_E) x=(x_F,g_E x_E g_E^{-1}, v_1 g_E^{-1}, v_{m-1}+v'_{m-1}).$$
One should note that $U_{m-1}\rtimes \mathrm{GL}(E)$ is not a normal subgroup of $U_P\rtimes \mathrm{GL}(E)$ and the action of the full $U_P$ is unfortunately very complicated in these coordinates. 
Nevertheless, we have a smooth surjective morphism
$$[(U_{m-1}\rtimes \mathrm{GL}(E)\backslash Y_m] \to [H_m\backslash Y_m].$$
This completes the proof of Theorem \ref{Bernstein-stratification}.
\end{proof}

\section{Action of $U_Q$ by left translation}

The unipotent radical $U_Q$ of the mirabolic group $Q$ consists of matrices of the form
$$u=
	\left[
	\begin{array}{c|c}
	1 & v \\ \hline
	0 & \mathrm{I}_{n-1}
	\end{array}
	\right]$$
where $v\in\mathrm{Hom}(E_1,F_1)$. The action of $U_Q$ on $G$ by left translation $g\mapsto ug$ respects the stratification $G=\bigsqcup_m X_m$. In this section, we will pay particular attention to the evaluation of the characteristic polynomial on left cosets of $U_Q$ i.e. the function $u\mapsto c(ux)$ for $x$ in each stratum $X_m$.

The characteristic polynomial $c(x)$ of $x\in G$, with formal variable $t$, is of the form $c(ux)=t^n+a_1 t^{n-1} +\cdots+a_{n}$ where $a_1,\ldots,a_n$ are $G$-invariant functions of $x$. It can be regarded as a morphism $c:G\to \mathbb{A}^n$ with $c(ux)=(a_1,\ldots,a_n)$.  

\begin{proposition}
For every $x\in G$, the morphism $l_x: U_Q=\mathrm{Hom}(E_1,F_1) \to \mathbb{A}^n$ given by $l_x(u)=c(ux)-c(x)$ is linear. If $x\in X_m$, the linear application $l_x$ is of rank $m-1$.	
\end{proposition}

For every $x\in X_m$ is $Q$-conjugate to a matrix of the form \eqref{x}
with $x_F$ being a companion matrix as in \eqref{companion}, we can assume 
that the matrix $x$ is of this special form. In particular we have $x\in P$ where 
$P$ the parabolic group which preserves the subspace $F_m$. Let $U_P$ and $L_P$ denote respectively its unipotent radical and the standard Levi component. 
Every element $u\in U_Q$ can be written uniquely in the form $u=u_L u_U$ where $u_L\in U_Q\cap L_P$ and $u_U\in U_Q \cap U_P$ 
where $u_L$ is a matrix of the form
\begin{equation} \label{U-L}
	u_L=\left[
	\begin{array}{c|c|c}
		1 & v & 0 \\ \hline
		0 & I_{m-1} & 0 \\ \hline
		0 & 0 & I_{n-m}	
	\end{array}
	\right]
\end{equation}
where $v=(v_1,\ldots,v_{m-1})\in \mathbb{A}^{m-1}$ is a row vector. The proposition can now be derived from a matrix calculation.

\begin{lemma} \label{charpoly}
For $x\in X_m$ of the form \eqref{x} with $x_F$ being a companion matrix as in \eqref{companion} and $u\in U_Q$ with $u=u_L u_U$ as above, we have
$$c(ux)=c(u_L x_F) c(x_E).$$
Moreover, if we write $u_L$ in coordinates $(v_1,\ldots,v_{m-1}) \in V=\mathbb{A}^{m-1}$ as above, and write $c(u_L x_F) -c_L(x_F)$ in coordinates $(a_1,\ldots,a_m) \in A=\mathbb{A}^m$ that are coefficients of the characteristic polynomials, then 
the application $u_L \mapsto c(u_L x_F) -c_L(x_F)$ induces an linear isomorphism between $V$ and the subspace of $A$ defined by the equation $a_m=0$.
\end{lemma}
 
\begin{proof} By direct calculation, we find the following formula for the
characteristic polynomial of 
$$u_L x_F=	\left[
	\begin{array}{c|ccc}
		1 & -v_1 & \cdots & -v_{m-1} \\ \hline
		0 &  & & \\
		\vdots & & I_{m-1} & \\ 
		0 & 
	\end{array}
	\right]
	\left[
	\begin{array}{ccc|c}
		0 & \cdots & 0 & -a_m \\ \hline
		  &  &         & -a_{m-1} \\
		  & \mathrm{I}_{m-1} & & \vdots \\
		  & & & -a_1
	\end{array}
	\right].$$
We have 
$$c(u_L x_F)=t^m+b_1 t^{m-1} +\cdots + b_{m-1}t+ b_m$$
where $a_m=b_m$ and for $1\leq r\leq m-1$
$$b_r=a_r+\sum_{i=1}^{r-1} a_i v_{r-i}+ v_r.$$ 
The lemma follows.
\end{proof}

\section{Proof of Theorem \ref{BKGL}}

We will deduce Theorem \ref{BKGL} from the analogous statement for the mirabolic subgroup $Q$ that ($\dagger$) if $g\in G-Q$, then $\mathrm{H}_c^*(U_Qg,\Phi_\rho|_{U_Qg})=0$.

To this end take $g\in G-B$, there are two cases: $g\in Q$ or $g\notin Q$.

If $g\notin Q$, by ($\dagger$) we know that $\mathrm{H}_c^*(U_Qg,\Phi_\rho|_{U_Qg})=0$. In fact for all $u\in U_B$, $ug\notin Q$ so that more generally we have $\mathrm{H}_c^*(U_Q u g,\Phi_\rho|_{U_Qug})=0$ for all $u\in U_B$. Now one can establish the vanishing of $\mathrm{H}_c^*(U_B g,i^*\Phi_\rho)$ by using the Leray spectral sequence associated with the morphism $U_Bg \to U_Q \backslash (U_Bg)$.

Now we consider the case $g\in Q$. Let $L_Q$ denote the standard Levi factor of $Q$, $g_L$ the image of $g$ in $L_Q$. We have $g_L \notin B\cap L_Q$. Using the definition of the restriction functor, we have
$$\mathrm{H}^*_c(U_B g,i^*\Phi_\rho)= \mathrm{H}^*_c((U_B\cap L)g_L, \mathrm{Res}_L^G(\Phi_\rho)|_{(U_B\cap L)g_L})$$
where $\mathrm{Res}_L^G(\Phi_\rho)= \Phi_{L,\rho|_{\check L}}$ by Proposition \ref{restriction lemma}. At this point we can conclude by an induction argument.

It remains to establish ($\dagger$). For convenience of induction we will prove the following equivalent proposition:
\begin{proposition}\label{vanishing-mirabolic}
Let $G$ be a direct product of general linear groups and $Q$ a mirabolic subgroup of $G$ of the form
$$Q=\prod_{i\neq j}\mathrm{GL}(n_i)\times Q_j\subset\prod_i\mathrm{GL}(n_i)=G,$$
let $\rho$ be a $\sigma$-positive representation of $\check{G}$ where $\sigma$ denotes the product of the characters $\mathrm{det}_i:\mathrm{GL}_{n_i}(\mathbb{C})\to\mathbb{C}^\times$. If $x$ is a geometric point of $G-Q$, then
$$\mathrm{H}_{c}^*(U_Qx,i^*\Phi_\rho)=0$$
where $i:U_Qx\to G$ denotes the inclusion map.
\end{proposition}

\begin{proof}
Argue by induction on the semisimple rank of $G$. In the base case $G$ is a torus, hence the proposition holds vacuously.

	Otherwise $n_j\geq2$, consider the stratification induced by \eqref{stratification} on $\mathrm{GL}(n_j)$: $$G=\bigsqcup_m X_m=\bigsqcup_{m=1}^{n_j} \bigg(\prod_{i\neq j}\mathrm{GL}(n_i)\times X_{j,m}\bigg).$$ For $x\notin Q$ with $Q=X_1$, we have $x\in X_m$ for $2\leq m \leq n_j$. 
	
	We first consider the case $x\in X_{n_j}$. For $X_{j,n_j}$ is contained in the open subset $\mathrm{GL}(n_j)^\mathrm{reg}$ of $\mathrm{GL}(n_j)$ (see \cite{S65}), we have a Cartesian diagram
$$\begin{tikzcd}
\tilde X_{n_j} \arrow{r}{\tilde c} \arrow{d}[swap]{\tilde q}
& \prod_{i\neq j}\mathrm{GL}(n_i)\times T_j \arrow{d}{q} \\
X_{n_j} \arrow{r}[swap]{c}
& \prod_{i\neq j}\mathrm{GL}(n_i)\times T_j/W_j
\end{tikzcd}$$
where $c$ is smooth and $q$ is finite.  It follows that the restriction of the $\gamma$-sheaf $\Phi_\rho$ to $X_{n_j}$ can be identified with a pullback by the characteristic polynomial map
$$\Phi_\rho|_{X_{n_j}}=c^* q_* \Phi_{\rho'}{}^{W_j}$$
where $\rho'$ denotes the restriction of $\rho$ to the subgroup $$\prod_{i\neq j}\mathrm{GL}_{n_i}(\mathbb{C})\times\check{T_j}\subset\check{G}.$$

We recall that the coordinate ring $T_j/W_j=\mathbb{A}^{n_j-1}\times \mathbb{G}_m$ is the ring of $\mathrm{GL}(n_j)$-invariant functions on $\mathrm{GL}(n_j)$, the projection $\sigma_{W_j}:T_j/W_j\to \mathbb{G}_m$ corresponds to the determinant function. 
By Lemma \eqref{charpoly}, the restriction of $c$ to $U_Q x$ induces an isomorphism between $U_Q x$ and the fiber of the determinant map on the $j$th component $$\sigma_{W_j}:\prod_{i\neq j}\mathrm{GL}(n_i)\times T_j/W_j \to \prod_{i\neq j}\mathrm{GL}(n_i)\times\mathbb{G}_m$$ over the image of $x$. 
Thus, to prove the proposition, it is enough to prove that
$$\sigma_!\Phi_{\rho'}{}^{W_j}= \sigma_{W_j,!}  q_! \Phi_{\rho'}{}^{W_j} =0$$
where $$\sigma:\prod_{i\neq j}\mathrm{GL}(n_i)\times T_j \to \prod_{i\neq j}\mathrm{GL}(n_i)\times\mathbb{G}_m$$ is the determinant map on the $j$th component.

Now recall the definition of the hypergeometric sheaf $\Psi_{\underline\lambda}=\mathrm{p}_{\underline\lambda,!} \mathrm{tr}^* \mathcal{L}_\psi$ with homorphism $\mathrm{p}_{\underline \lambda}:\mathbb{G}_m^r\to T$ given by  
$$\mathrm{p}_{\underline\lambda}(t_1,\ldots,t_r)=\prod_{i=1}^r \lambda_i(t_i).$$
We have
$$\sigma_!\Phi_{\rho'} = \sigma_! \mathrm{Ind}_T^{G'} \Psi_{\underline\lambda} {}^{\prod_{i\neq j}W_i}$$
where $G'=\prod_{i\neq j}\mathrm{GL}(n_i)\times T_j$.

With the definition of the action of $W$ on $\Psi_{\underline\lambda}$ given by \eqref{action-of-W}, we see that the induced action of $W$ on $\sigma_! \Phi_{\rho'}$ is through the character $\mathrm{sign}_{n_j}:W_j\to \{\pm 1\}$. It follows that $\sigma_! \Phi_{\rho'}{}^{W_j}=0$ because for $n_j\geq 2$ the sign character $\mathrm{sign}_{n_j}$ is nontrivial. This concludes the case $x\in X_{n_j}$.

We now consider the general case $x\in X_m$ with $2\leq m\leq n_j$. By $Q$-conjugation we can assume that the $j$th component $x_j$ of $x$ is of the form \eqref{x} with $x_{j,F}$ being a companion matrix as in \eqref{companion}. Let $P$ denote the standard parabolic of block matrices as in \eqref{x}, $L$ the standard Levi factor of $P$ and $U_P$ its unipotent radical. By applying the result obtained above in the generic case to $\mathrm{GL}(F)$, we get
\begin{equation} \label{vanishing-Levi}
	\mathrm{H}^*_c(U_L x_L,\Phi_{L,\rho|_{\check L}}|_{U_L x_L})=0
\end{equation}
where $x_L$ is the image of $x$ in $L$, $U_L$ consists of unipotent matrices of the form \eqref{U-L}, and $\Phi_{L,\rho|_{\check L}}$ is the $\gamma$-sheaf on $L$ associated to the restriction to $\check L$ of the representation $\rho$ of $\check G$. 

For $\Phi_{L,\rho|_{\check L}}=\mathrm{Res}_L^G (\Phi_\rho)$ by Proposition \ref{restriction lemma}, \eqref{vanishing-Levi} implies that
\begin{equation} \label{vanishing-restriction}
	\mathrm{H}^*_c(U_P U_L x,\Phi_{\rho}|_{U_P U_L x})=0
\end{equation}
where $U_P$ and $U_L$ commute. With the help of Lemma \ref{unipotent-conjugation}, we see that the  morphism
$$U_{Q} \times U_{m-1} \to U_P U_L x$$ 
given by
$$(u_{m-1},u_Q) \mapsto u_{m-1}u_Q x u_{m-1}^{-1}$$
is an isomorphism. Now using the fact that $\Phi_\rho$ is equivariant under the adjoint action, \eqref{vanishing-restriction} implies that
$$\mathrm{H}^*_c(U_Q x,\Phi_{\rho}|_{U_Q x})=0.$$
This concludes the proof of Proposition \ref{vanishing-mirabolic} and therefore Theorem \ref{BKGL}.
\end{proof}

\appendix

\section{The case of parabolic subgroups}

In this appendix we give a proof for the extension of Theorem \ref{BKGL} to arbitrary parabolic subgroups $P$ of $G=\mathrm{GL}(n)$. By a similar argument involving the Leray spectral sequence as in the beginning of the proof of Theorem \ref{BKGL}, we are reduced to the case when $P$ is a maximal parabolic subgroup.

\begin{proposition} 
Let $P$ denote the maximal standard parabolic subgroup of $G=\mathrm{GL}(n)$ consisting of block matrices of size $(n_1,n_2)$ where $n_1+n_2=n$ and $U_P$ its unipotent radical, let $\rho$ be a $\sigma$-positive representation of $\check{G}$ where $\sigma=\mathrm{det}$. If $g$ is a geometric point of $G-P$, then
$$\mathrm{H}^*_c(U_Pg,i^*\Phi_\rho)=0$$
where $i:U_Pg\to G$ denotes the inclusion map.
\end{proposition}

\begin{proof}
Argue by induction on $n_1$. In the base case when $n_1=1$ we are reduced to Proposition \ref{vanishing-mirabolic}.

Otherwise $n_1\geq2$, let $P'$ denote the maximal standard parabolic subgroup of $\mathrm{GL}(n)$ consisting of block matrices of size $(n_1-1,n_2+1)$ and $U_{P'}$ its unipotent radical. Let $$g'=\left[\begin{array}{c|c}
    a&b\\\hline c&d
\end{array}\right]$$ be a block matrix of size $(n_1,n_2)$ such that the $n_2\times n_1$-matrix $c$ is of the form $$b=\left[\begin{array}{ccc|c}
    * & \cdots & * & 1 \\\hline
    * & \cdots & * & 0 \\
    \vdots & \ddots & \vdots & \vdots \\
    * & \cdots & * & 0
\end{array}\right],$$ then by an analogous computation as in the proof of Lemma \ref{unipotent-conjugation} and the end of the proof of Proposition \ref{vanishing-mirabolic}, we see that acyclicity of $\Phi_\rho$ over $U_{P'}g$:
\begin{equation}\label{vanishing-maximal-parabolic}
    \mathrm{H}_c^*(U_{P'}g',\Phi_\rho|_{U_{P'}g'})=0
\end{equation}
implies acyclicity over the subcoset $(U_P\cap U_{P'})g$: $$\mathrm{H}_c^*\big((U_P\cap U_{P'})g',\Phi_\rho|_{(U_P\cap U_{P'})g'}\big)=0,$$ which then implies the proposition by the Levy spectral sequence.

Now it remains to prove \eqref{vanishing-maximal-parabolic}. To this end it suffices to observe that $g$ is conjugate, under the action of the standard Levi factor $L$ of $P$, to a matrix of the form $g'$, which is in addition not contained in $P'$. This follows from the fact that the $L$-orbits on $c$ are classified by the rank of $c$. Hence we are done by induction.
\end{proof}

\section{Positive hypergeometric sheaves}\label{appendix-positive-hypergeometric-sheaves}

In this appendix we give a proof for Proposition \ref{positive-hypergeometric-sheaves} concerning hypergeometric sheaves $\Psi_{\underline{\lambda}}$ when $\underline{\lambda}$ is $\sigma$-positive. The first half of Proposition \ref{positive-hypergeometric-sheaves} is due to Braverman--Kazhdan in \cite[Theorem 4.2]{BK02}:

\begin{proposition}\label{cleanness lemma}
If $\underline{\lambda}$ is $\sigma$-positive, then the forget support morphism $\Psi_{\underline{\lambda}}\to\Psi_{\underline{\lambda}}^*$ is an isomorphism.
\end{proposition}

This is proved by restricting to the smooth neighborhood $\mathrm{p}_{\underline{\lambda}}:\mathbb{G}_m^r\to T$ and then applying the classical Fourier--Deligne transform on $\mathbb{G}_a^r$. The same idea is also crucial in the proof of the second half of Proposition \ref{positive-hypergeometric-sheaves}:

\begin{proposition}
If $\underline{\lambda}$ is $\sigma$-positive, then $\Psi_{\underline{\lambda}}$ is isomorphic to a shift of a local system on the image of $\mathrm{p}_{\underline{\lambda}}:\mathbb{G}_m^r\to T$.
\end{proposition}

Without loss of generality we can assume that $\mathrm{p}_{\underline{\lambda}}$ is surjective. Then we can also factorize $\mathrm{p}_{\underline{\lambda}}$ into the product of a homomorphism with connected fibers and an isogeny. Since isogenies preserve local systems under pushforward by proper-smooth base change, without loss of generality we can assume that $\mathrm{p}_{\underline{\lambda}}$ has connected fibers.

We will deduce smoothness of $\Psi_{\underline{\lambda}}$ from universal local acyclicity by Th\'{e}or\`{e}m 5.3.1 in \cite{D77a} together with the fact that $\Psi_{\underline{\lambda}}$ is a perverse sheaf by Proposition \ref{cleanness lemma}, for this we need a compactification of $\mathrm{p}_{\underline{\lambda}}$. Let $\overline{\Gamma}$ be the normalization of the closure of the graph of $\mathrm{p}_{\underline{\lambda}}$ in $(\mathbb{G}_m)^r\times T\subset(\mathbb{P}^1)^r\times T$, let $\mathrm{j}_{\underline{\lambda}}:(\mathbb{G}_m)^r\to\overline{\Gamma}$ be the open inclusion and $\overline{\mathrm{p}}_{\underline{\lambda}}:\overline{\Gamma}\to T$ the projection, then $\overline{\mathrm{p}}_{\underline{\lambda}}$ compactifies $\mathrm{p}_{\underline{\lambda}}$.

We are therefore reduced to the following lemma:

\begin{lemma}
If $\underline{\lambda}$ is $\sigma$-positive and $\mathrm{p}_{\underline{\lambda}}$ is surjective with connected fibers, then $\overline{\mathrm{p}}_{\underline{\lambda}}$ is universally locally acyclic with respect to $\mathrm{j}_{\underline{\lambda},!}\mathrm{tr}^*\mathcal{L}_\psi$.
\end{lemma}

\begin{proof}
The following argument is essentially due to Katz--Laumon in \cite{KL85}.

Let $\pi:(\mathbb{P}^1)^r\to(\mathbb{P}^1)^r$ be the completed Artin--Schreier covering defined by $[X:Y]\mapsto[X^q-XY^{q-1}:Y^q]$ on the homogeneous coordinates on each factor $\mathbb{P}^1$ and $\tilde{\pi}$ the base change of $\pi$ along the projection $\overline{\Gamma}\to(\mathbb{P}^1)^r$. It follows from the decomposition theorem and total ramification of $\pi$ at infinity that the $!$-extension of $\mathrm{tr}^*\mathcal{L}_\psi$ to $(\mathbb{P}^1)^r$ is a direct summand of $\pi_*\overline{\mathbb{Q}}_\ell$ over $(\mathbb{P}^1)^r-(0,\dots,0)$.

Now by assumption $\underline{\lambda}$ is $\sigma$-positive, hence the image of $\overline{\Gamma}$ in $(\mathbb{P}^1)^r$ is contained in $(\mathbb{P}^1)^r-(0,\dots,0)$ which implies that $\mathrm{j}_{\underline{\lambda},!}\mathrm{tr}^*\mathcal{L}_\psi$ is a direct summand of $\tilde{\pi}_*\overline{\mathbb{Q}}_\ell$. Therefore it suffices to show instead that $\overline{\mathrm{p}}_{\underline{\lambda}}$ is universally locally acyclic with respect to $\tilde{\pi}_*\overline{\mathbb{Q}}_\ell$.

Let $K$ be the kernel of $\mathrm{p}_{\underline{\lambda}}$ which we may assume to be connected, let $\overline{K}$ be the normalization of the closure of $K$ in $(\mathbb{P}^1)^r$ and $\overline{\mu}:(\mathbb{G}_m)^r\times\overline{K}\to\overline{\Gamma}$ the morphism which extends the multiplication map from $(\mathbb{G}_m)^r\times K$ to $(\mathbb{G}_m)^r$. Then the diagram $$
\begin{tikzcd}
(\mathbb{G}_m)^r\times\overline{K} \arrow{r}{\overline{\mu}} \arrow{d}[swap]{pr}
& \overline{\Gamma} \arrow{d}{\overline{\mathrm{p}}_{\underline{\lambda}}} \\
(\mathbb{G}_m)^r \arrow{r}[swap]{\mathrm{p}_{\underline{\lambda}}}
& T
\end{tikzcd}
$$ is Cartesian by the theory of toric varieties (see \cite{F93} for example). Let $\hat{\pi}$ be the base change of $\tilde{\pi}$ along $\overline{\mu}$.

We may assume that $\mathrm{p}_{\underline{\lambda}}$ is surjective hence smooth. Since universal local acyclicity is a local property with respect to the smooth topology, it suffices to show that $pr$ is universally locally acyclic with respect to $\hat{\pi}_*\overline{\mathbb{Q}}_\ell$. Then by properness of $\hat{\pi}$ we are further reduced to showing that the composite $pr\circ\hat{\pi}$ is universally locally acyclic with respect to the constant sheaf $\overline{\mathbb{Q}}_\ell$ on the source.

Let $U$ be an open subscheme of $(\mathbb{P}^1)^r$ of the form $U=U_1\times\dots\times U_r$ where each Cartesian factor $U_i$ is an open subscheme of $\mathbb{P}^1$ equal to either $\mathbb{G}_a$ or $\mathbb{P}^1-0$. Let $\widetilde{U}$ be the inverse image of $U$ under $\pi:(\mathbb{P}^1)^r\to(\mathbb{P}^1)^r$ and $\overline{K}_U$ the intersection of $\overline{K}$ with $U$. We will check universal local acyclicity of $pr\circ\hat{\pi}$ by restricting to the open subscheme $\widetilde{U}\times_U((\mathbb{G}_m)^r\times\overline{K}_U)$ of the source of $\hat{\pi}$ and then calculating ``by hand'' with explicit coordinates.

To this end let $t$ be a coordinate on $\mathbb{G}_m$, $X$ a coordinate on $\mathbb{G}_a$ and $Y^{-1}$ a coordinate on $\mathbb{P}^1-0$, hence $(\dots,X_i,\dots,Y_j{}^{-1},\dots)$ are coordinates on $U$. By slight abuse of notation let $\mathrm{p}_{\underline{\lambda}}(\dots,X_i,\dots,Y_j{}^{-1},\dots)=1$ denote a system of polynomial equations such that $$\overline{K}_U\simeq\mathrm{Spec}\Bigg(\overline{\mathrm{Int}}\bigg(\frac{k[\dots,X_i,\dots,Y_j,\dots]}{\big(\mathrm{p}_{\underline{\lambda}}(\dots,X_i,\dots,Y_j{}^{-1},\dots)=1\big)}\bigg)\Bigg)$$ where $\overline{\mathrm{Int}}$ denotes the integral closure, then modulo base change with respect to a normalization, we have that $\widetilde{U}\underset{U}\times\big((\mathbb{G}_m)^r\times\overline{K}_U\big)$ is isomorphic to $$\mathrm{Spec}\left(\frac{k\bigg[\begin{array}{c}\dots,t_n,t_n{}^{-1},\dots,X_i,\dots,Y_j,\dots\\\dots,\widetilde{X}_i,\dots,\widetilde{Y}_j,(1-\widetilde{Y}_j{}^{q-1})^{-1},\dots\end{array}\bigg]}{\bigg(\begin{array}{c}\widetilde{X}_i{}^q-\widetilde{X}_i=t_iX_i,(1-\widetilde{Y}_j{}^{q-1})^{-1}\widetilde{Y}_j{}^q=t_j{}^{-1}Y_j,\\\mathrm{p}_{\underline{\lambda}}(\dots,X_i,\dots,Y_j{}^{-1},\dots)=1\end{array}\bigg)}\right),$$ which by substituting $s_j$ for $(1-\widetilde{Y}_j{}^{q-1})^{-1}t_j$ becomes isomorphic to $$\mathrm{Spec}\left(\frac{k\bigg[\begin{array}{c}\dots,t_i,t_i{}^{-1},\dots,s_j,s_j{}^{-1},\dots\\\dots,\widetilde{X}_i,\dots,\widetilde{Y}_j,(1-\widetilde{Y}_j{}^{q-1})^{-1},\dots\end{array}\bigg]}{\Big(\mathrm{p}_{\underline{\lambda}}\big(\dots,t_i{}^{-1}(\widetilde{X}_i{}^q-\widetilde{X}_i),\dots,\big(s_j\widetilde{Y}_j{}^q\big){}^{-1},\dots\big)=1\Big)}\right).$$ The morphism $\widetilde{X}\mapsto\widetilde{X}^q-\widetilde{X}$ is \'{e}tale and the morphism $\widetilde{Y}\mapsto\widetilde{Y}^q$ is finite surjective radicial, both are universally locally acyclic with respect to $\overline{\mathbb{Q}}_\ell$ on the source. By Corollaire 2.16 in \cite{D77f} the projection $pr$ from $$\mathrm{Spec}\Bigg(\overline{\mathrm{Int}}\bigg(\frac{k[\dots,t_n,t_n{}^{-1},\dots,X_i,\dots,Y_j,\dots]}{\big(\mathrm{p}_{\underline{\lambda}}(\dots,t_i{}^{-1}X_i,\dots,(t_jY_j){}^{-1},\dots)=1\big)}\bigg)\Bigg)\simeq(\mathbb{G}_m)^r\times\overline{K}_U$$ to $(\mathbb{G}_m)^r$ is universally locally acyclic with respect to $\overline{\mathbb{Q}}_\ell$ on the source. Hence so is their composite, the lemma follows.
\end{proof}


\begin{thebibliography}{LABEL}
\addcontentsline{toc}{section}{References}

\bibitem[BBD82]{BBD82} A. Beilinson, J. Bernstein and P. Deligne, \emph{Faisceaux pervers}, in Analysis and topology on singular spaces, I (Luminy 1981), Ast\'{e}risque, 100 (1982), 5-171.

\bibitem[B84]{B84} J. Bernstein, \emph{P-invariant distributions on GL(N) and the classification of unitary representations of GL(N) (nonArchimedean case)}. in Lie group representations, II (College Park, Md., 1982/1983), 50–102, Lecture Notes in Math., 1041, Springer, Berlin, 1984.

\bibitem[BK00]{BK00} A. Braverman and D. Kazhdan, \emph{$\gamma$-functions of representations and lifting} (with an appendix by V. Vologodsky), Geom. Funct. Anal., Special Volume (2002), Part I, 237-278.

\bibitem[BK02]{BK02} A. Braverman and D. Kazhdan, \emph{$\gamma$-sheaves on reductive  groups}, in Studies in memory of Issai Schur (Chevaleret/Rehovot 2000), Progr. Math., 210 (2003), 27-47.

\bibitem[D77a]{D77a} P. Deligne, \emph{Cohomologie \'{e}tale: les points de d\'{e}part}, in Cohomologie \'{e}tale (SGA 4 1/2), Lecture Notes in Math., 569 (1977), 4-75.

\bibitem[D77s]{D77s} P. Deligne, \emph{Applications de la formule des traces aux sommes trigonom\'{e}triques}, in Cohomologie \'{e}tale (SGA 4 1/2), Lecture Notes in Math., 569 (1977), 168-232.

\bibitem[D77f]{D77f} P. Deligne, \emph{Th\'{e}or\`{e}mes de finitude en cohomologie $\ell$-adique}, in Cohomologie \'{e}tale (SGA 4 1/2), Lecture Notes in Math., 569 (1977), 233-261.

\bibitem[F93]{F93} W. Fulton, \emph{Introduction to toric varieties}, Ann. of Math. Stud., 131 (1993).

\bibitem[GJ72]{GJ72} R. Godement and H. Jacquet, \emph{Zeta functions of simple algebras}, Lecture Notes in Math., 260 (1972).

\bibitem[KL85]{KL85} N. Katz and G. Laumon, \emph{Transformation de Fourier et majoration de sommes exponentielles}, Publ. Math. IHES, 62 (1985), 145-202.

\bibitem[K90]{K90} N. Katz, \emph{Exponential sums and differential equations}, Ann. of Math. Stud. 124(1900), 1-430.

\bibitem[L13]{L13} L. Lafforgue, \emph{Noyaux du transfert automorphe de Langlands et formules de Poisson non lin\'{e}aires}, Preprint, 2013, \url{http://www.ihes.fr/~lafforgue/publications.html}.

\bibitem[T50]{T50} J. Tate, \emph{Fourier analysis in number fields, and Hecke's zeta-functions}, in \emph{Algebraic Number Theory}, Academic Press, New York, 1967, 305-347.

\bibitem[S65]{S65} R. Steinberg, \emph{Regular elements of semisimple algebraic groups}, Publ. Math. IHES, 25 (1965), 49-80.

\end{thebibliography}
\end{document}